\documentclass[12pt,reqno]{amsart}
\makeatletter
\def\subsection{\@startsection{section}{1}%
  \z@{1.2\linespacing\@plus\linespacing}{.5\linespacing}%
  {\normalfont\scshape\centering}}
\def\subsection{\@startsection{subsection}{2}%
  \z@{0.9\linespacing\@plus.7\linespacing}{-.5em}%
  {\normalfont\bfseries}}
\makeatother

\usepackage{amsmath}
\usepackage{amsfonts}
\usepackage{amsbsy}
\usepackage{amssymb}
\usepackage[normalem]{ulem}
\usepackage{times}
\usepackage{mathrsfs}
\usepackage{exscale}
\usepackage{graphicx}
\usepackage[colorlinks,citecolor=blue,linkcolor=rouge,
            bookmarksopen,
            bookmarksnumbered
           ]{hyperref}

\usepackage{times}

\usepackage{color}
\definecolor{gr}{rgb}   {0.,   0.6,   0.25 }
\definecolor{mg}{rgb}   {0.85,  0.,    0.85}
\definecolor{marin}{rgb}   {0.,   0.,   0.8}
\definecolor{rouge}{rgb}   {0.8,   0.,   0.}
\definecolor{orange}{rgb}   {0.8,   0.4,   0.}

\newtheorem{theorem}{Theorem}[section]
\newtheorem{lemma}[theorem]{Lemma}
\newtheorem{proposition}[theorem]{Proposition}
\newtheorem{corollary}[theorem]{Corollary}

\theoremstyle{definition}

\theoremstyle{remark}
\newtheorem{remark}[theorem]{Remark}

\numberwithin{equation}{section}

\newcommand{\dd}[1]{_{\raise-0.3ex\hbox{$\scriptstyle #1$}}}
\newcommand{\di}{\displaystyle}
\newcommand{\on}[1]{\raise-.5ex\hbox{\big|}_{#1}}

\newcommand{\eps}{\varepsilon}

\renewcommand{\div}{\operatorname{\rm div}}

\newcommand{\curl}{\operatorname{\rm curl}}

\newcommand\C{{\mathbb C}}
\newcommand\R{{\mathbb R}}
\newcommand\N{{\mathbb N}}

\newcommand\T{{\mathbb T}}
\newcommand\Z{{\mathbb Z}}

\usepackage{soul}

\begin{document}

\title[Limit Sobolev regularity for  Dirichlet and Neumann problems on Lipschitz domains]{On the limit Sobolev regularity for  Dirichlet and Neumann problems on Lipschitz domains}
\author{Martin Costabel}
\date{2017-11-19}

\maketitle

\begin{abstract}
We construct a bounded $C^{1}$ domain $\Omega$ in $\R^{n}$ for which the $H^{3/2}$ regularity for the Dirichlet and Neumann problems for the Laplacian cannot be improved, that is, there exists $f$ in $C^{\infty}(\overline\Omega)$ such that the solution of  $\Delta u=f$ in $\Omega$ and either $u=0$ on $\partial\Omega$ or $\partial_{n} u=0$ on $\partial\Omega$ is contained in $H^{3/2}(\Omega)$ but not in $H^{3/2+\varepsilon}(\Omega)$ for any $\epsilon>0$. An analogous result holds for $L^{p}$ Sobolev spaces with $p\in(1,\infty)$.  
\end{abstract}

\section{Introduction}\label{S:intro}
The motivation for this note comes from a question of regularity of the time-harmonic Maxwell equations in Lipschitz domains.
In the variational theory of Maxwell's equations, basis for the analysis of many algorithms of numerical electrodynamics, the following two function spaces are fundamental:
\begin{align}
\nonumber
 X_{N} &= H(\div,\Omega)\cap H_{0}(\curl,\Omega) \\
\label{E:XNdef}
  &=\{u \in L^{2}(\Omega;\C^{3}) \mid
      \div u \in L^{2}(\Omega), \curl u \in L^{2}(\Omega;\C^{3}), 
      u\times n =0 \text{ on }\partial\Omega\}\\
\nonumber
 X_{T} &= H_{0}(\div,\Omega)\cap H(\curl,\Omega) \\
\label{E:XTdef}
  &=\{u \in L^{2}(\Omega;\C^{3}) \mid
      \div u \in L^{2}(\Omega), \curl u \in L^{2}(\Omega;\C^{3}),
      u\cdot n =0 \text{ on }\partial\Omega \}
\end{align} 
Here $n$ is the outward unit normal vector field on the boundary of the domain $\Omega\subset\R^{3}$.

If $\Omega$ is a bounded Lipschitz domain, then it has been known for a long time \cite{Weber,Picard1984} that  $X_{N}$ and $X_{T}$ are compactly embedded subspaces of $L^{2}(\Omega;\C^{3})$, and it has been shown more precisely \cite{Co_maxreg,Mitrea2002} that they are contained in the Sobolev space
$H^{\frac12}(\Omega,\C^{3})=W^{\frac12,2}(\Omega,\C^{3})$.
For large classes of more regular domains, $X_{N}$ and $X_{T}$ are contained in $H^{1}(\Omega,\C^{3})$ (see \cite{Amrouche-et-al} for $C^{1,1}$ domains, \cite{Filonov97} for $C^{\frac32+\eps}$ domains, \cite{Saranen1982} for convex domains, \cite{Taylor-Mitrea-Vasy} for ``almost convex'' domains). 
The regularity is diminished by corner singularities, but one also knows \cite{Amrouche-et-al} that for every Lipschitz polyhedron or, more generally, piecewise smooth domain $\Omega$ that is at least $C^{2}$-diffeomorphic to a polyhedron, there exists $\eps>0$ such that
\begin{equation}
\label{E:eps}
 X_{N} \cup X_{T} \subset H^{\frac12+\eps}(\Omega;\C^{3}) \,.
\end{equation}
The additional regularity described by $\eps$ is of some use in the numerical analysis of Maxwell's equations (see for example \cite{AlonsoValli1999,AinsworthGuzmanSayas2016}).
The parameter $\eps$ can become arbitrarily small, depending on the corner angles of $\partial\Omega$, but it depends only on these angles, that is, on the local Lipschitz constant of $\partial\Omega$. 
Based on this observation, one could ask the question whether for any Lipschitz domain $\Omega$, there exists such an $\eps>0$ for which \eqref{E:eps} holds. 
This question is the motivation for the present investigation.

To the best of the author's knowledge, the conjecture that such an $\eps>0$ always exists is not incompatible with the currently available regularity results for Maxwell's equations on Lipschitz domains, but we shall show that it is not true. As a corollary of our constructions, we obtain a counterexample that is even $C^{1}$.

\begin{proposition}
\label{P:XNXT}
There exists a bounded $C^{1}$ domain $\Omega\subset\R^{3}$, an $L^{2}(\Omega)$ function $g$ and an $L^{2}(\Omega;\C^{3})$ function $h$ such that the solutions $u\in L^{2}(\Omega;\C^{3})$ of the system
\begin{equation}
\label{E:divcurl}
 \div u = g\,,\qquad \curl u = h \quad\mbox{ in }\Omega
\end{equation} 
and either
\begin{equation}
\label{E:XNreg}
 u\times n = 0 \quad\mbox{ on }\partial\Omega
\end{equation}
or
\begin{equation}
\label{E:XTreg}
 u\cdot n = 0 \quad\mbox{ on }\partial\Omega
\end{equation}
do not belong to $H^{\frac12+\eps}(\Omega;\C^{3})$ for any $\eps>0$.\\
In the system \eqref{E:divcurl}, the field $h$ can be chosen to be zero and $g$ can be chosen to be continous on $\overline\Omega$.
\end{proposition}
As we will see in the following, analogous results are true in dimension $2$ and in higher dimensions, and also for non-Hilbert Sobolev spaces over $L^{p}$ with $p$ different from $2$. 

Non-regular solutions of the div-curl system \eqref{E:divcurl} are typically sought as gradients of solutions of the inhomogeneous Laplace (Poisson) equation with either Dirichlet (for \eqref{E:XNreg}) or Neumann (for \eqref{E:XTreg}) boundary conditions. A non-regularity result for these Laplace boundary value problems is the main result of this paper, see Theorem~\ref{T:main} below. It will be  
proved in Section~\ref{S:d=2} for dimension $d=2$ and in Section~\ref{S:d>2} for higher dimensions.

We use the standard notation $W^{s,p}(\Omega)$ for the Sobolev-Slobodeckij spaces on $\Omega\subset\R^{d}$, and we recall that for $0<s<1$ the seminorm
\begin{equation}
\label{E:sp-seminorm}
 |u|_{s,p;\Omega} =
 \left(\int_{\Omega}\int_{\Omega}\frac{|u(y)-u(x)|^{p}}{|y-x|^{d+sp}}dx\,dy
 \right)^{\frac1p} 
\end{equation}
defines the norm 
$\|u\|_{W^{s,p}(\Omega)}=\|u\|_{L^{p}(\Omega)}+|u|_{s,p;\Omega}$,
that $W^{0,p}(\Omega)=L^{p}(\Omega)$,
 and that for any $s$ there holds
$$
u\in W^{s+1,p}(\Omega)\iff u\in W^{s,p}(\Omega) \mbox{ and }\nabla u\in W^{s,p}(\Omega;\C^{d})\,.
$$

In order to describe known regularity results, we also need the Bessel potential spaces $H^{s,p}(\Omega)$, which are different from $W^{s,p}(\Omega)$ if $p\ne2$. For the main properties of these spaces, see \cite{Triebel1978}. In Triebel's notation
$W^{m,p}(\Omega) = F^{m}_{p,2}(\Omega)$ for $m\in\N$ and
$$
 H^{s,p}(\Omega) = F^{s}_{p,2}(\Omega)\,, \quad\mbox{ and for $s\not\in\Z$ : }\;
 W^{s,p}(\Omega) = B^{s}_{p,p}(\Omega)\,. 
$$ 
Note that the trace space for both $W^{s,p}(\Omega)$ and $H^{s,p}(\Omega)$ on a sufficiently smooth boundary is $W^{s-\frac1p,p}(\partial\Omega)$ if $s>\tfrac1p$.
 
Comprehensive regularity results in the $H^{s,p}$ spaces for the Dirichlet and Neumann problems on Lipschitz domains were given by Jerison and Kenig \cite{JerisonKenig1981Neu,JerisonKenig1995}. In particular they studied the question for which $s$ and $p$ the condition $g\in H^{s-2,p}(\Omega)$ implies $v\in H^{s,p}(\Omega)$ for the solutions $v$ of the problems
\begin{align}
\label{E:LapDir}
 \Delta v &= g \quad\mbox{ in }\Omega\,,
 &v&=0 \quad\mbox{ on }\partial\Omega\\
\label{E:LapNeu}
 \Delta v &= g \quad\mbox{ in }\Omega\,,
 &\tfrac{\partial v}{\partial n} &=0 \quad\mbox{ on }\partial\Omega
\end{align} 
For the maximal regularity one finds a limit at $s=1+\tfrac1p$. We summarize the main results pertaining to the question of maximal regularity (here formulated for the Dirichlet problem, see \cite[Thms 1.1--1.3]{JerisonKenig1995}, where $H^{s,p}$ is written $L_{s}^{p}$ ):  

For any bounded Lipschitz domain $\Omega\subset\R^{d}$, $d\ge2$, there exists $p_{0}\ge1$ such that for $p_{0}<p<\frac{p_{0}}{p_{0}-1}$ and $\frac1p<s<1+\frac1p$ the solution $v$ of the Dirichlet problem \eqref{E:LapDir} with $g\in H^{s-2,p}(\Omega)$ belongs to $H^{s,p}(\Omega)$. In general, $p_{0}>1$ and there are counterexamples as soon as $p$ or $s$ are outside of the given bounds, but when $\Omega$ is a $C^{1}$ domain, one can choose $p_{0}=1$. 
When $p>2$, there are Lipschitz counterexamples with $g\in C^{\infty}(\overline\Omega)$ and $v\not\in W^{1+\frac1p,p}(\Omega)$.
There is a $C^{1}$ counterexample for $p=1$ with $g\in C^{\infty}(\overline\Omega)$ and $v\not\in W^{2,1}(\Omega)$.
In the optimal regularity-shift result for $C^{1}$ domains, the condition on $s$ cannot be weakened, because for any $p>1$ there exists a bounded $C^{1}$ domain $\Omega$ and a $g\in H^{-1+\frac1p,p}(\Omega)$ such that $v\not\in H^{1+\frac1p,p}(\Omega)$. On the other hand, if $g$ is more regular, for example $g\in H^{-1+\frac1p+\eps,p}(\Omega)$ for some $\eps>0$ and $p>1$, then $v\in H^{1+\frac1p,p}(\Omega)$ follows. 
The latter result is obtained by subtracting from $v$ a solution $v_{0}\in H^{1+\frac1p+\eps,p}(\Omega)$ of $\Delta v_{0}=g$ without boundary conditions and observing that a harmonic function with trace in $W^{1,p}(\partial\Omega)$ belongs to $H^{1+\frac1p,p}(\Omega)$.

We will prove that one cannot have $v\in H^{1+\frac1p+\eps,p}(\Omega)$ for any $\eps>0$, in general, even for more regular $g$. Because of the mutual inclusions 
$H^{s+\eps,p}\subset W^{s,p}\subset H^{s-\eps,p}$ for any $\eps>0$, the result is equivalently formulated in the scale of $W^{s,p}$ spaces.

\begin{theorem}
\label{T:main}
 In $\R^{d}$, $d\ge2$, there exists a bounded $C^{1}$ domain $\Omega$ and for both the Dirichlet problem \eqref{E:LapDir} and the Neumann problem \eqref{E:LapNeu} functions $g\in L^{\infty}(\Omega)$ such that the solutions 
 $v\in H^{1}(\Omega)$ do not belong to $W^{1+\frac1p+\eps,p}(\Omega)$ for any $p\in[1,\infty)$ and any $\eps>0$.
\end{theorem}
\begin{remark}
\label{R:gregular}
It will follow from the proof that in dimension $d=2$, there are functions $g\in C^{\infty}(\overline\Omega)$ that provide examples, even $g=1$ is possible for the Dirichlet problem and a second degree polynomial $g$ for the Neumann problem. See also Remark~\ref{R:DirEF}.
In dimension $d\ge3$, there is still an example with $g=1$ for the Dirichlet problem, and examples with $g\in C^{\alpha}(\overline\Omega)$, $\alpha>0$,  for the Neumann problem.
\end{remark}

\begin{remark}
\label{R:known}
Not all of this is new: 
For $p=1$, the counterexample from \cite[Theorem~1.2(b)]{JerisonKenig1995} shows that the result for the Dirichlet problem holds even with $\eps=0$.
Moreover, for $p>2$ the result of Theorem~\ref{T:main} is not interesting in the class of Lipschitz domains, because singularities at conical points provide a limit of regularity that is strictly below $s=1+\frac1p$. But for $C^{1}$ domains the result still seems to be new even for $p>2$. We provide a proof that works for any $p\ge1$, because there is no extra cost with respect to the proof for $p=2$. One just has to be careful to observe that the same domain $\Omega$ and the same function $g$ give an example valid for all $p$ and all $\eps$.
\end{remark}

Proposition~\ref{P:XNXT} follows from Theorem~\ref{T:main} for $p=2$, $d=3$ if we take $u=\nabla v$ (``electrostatic field''). The Laplace equation for $v$ implies the div-curl system \eqref{E:divcurl} for $u$ with $h=0$, and the Dirichlet and Neumann conditions in \eqref{E:LapDir} and \eqref{E:LapNeu} for $v$ imply the vanishing of the tangential component \eqref{E:XNreg} or of the normal component \eqref{E:XTreg}, respectively. Finally, $v\in W^{1+\frac1p+\eps,p}(\Omega)$ is equivalent to 
$u\in W^{\frac1p+\eps,p}(\Omega;\C^{3})$.

The construction of our counterexample uses the ideas of Filonov in the paper \cite{Filonov97}, where he considers a related question for $\eps=\frac12$ and constructs a $C^{\frac32}$ domain $\Omega$ that satisfies, among other interesting properties 
\[
 H^{2}(\Omega) \cap H^{1}_{0}(\Omega) = H^{2}_{0}(\Omega)\,,
\]
that is, the homogeneous Dirichlet condition for $H^{2}$ functions implies the homogeneous Neumann condition,
see also \cite{BuCoSh02}. 
Generalizing this, the $C^{1}$ domain $\Omega$ that we will construct satisfies
\begin{equation}
\label{E:Dir=>Neu}
 W^{1+\frac1p+\eps,p}(\Omega)\cap W_{0}^{1,p}(\Omega) = 
 W_{0}^{1+\frac1p+\eps,p}(\Omega)
 \quad \forall 1\le p<\infty\,,\eps>0\,.
\end{equation} 
 
\section{Generalizing Filonov's separating function}\label{S:FilonovEps}
We construct a continuous real-valued function $f$ on $\T=\R/(2\pi\Z)$ with the following property: 
If $a$ and $b$ belong to $W^{\varepsilon,p}(\T)$ for some $\epsilon>0$, $p\ge1$, and
$af=b$, then $a=b=0$.

The construction and proof are modeled after Filonov's construction of a $C^{\frac12}$ function that has the above separation property for $\varepsilon=\frac12$ and $p=2$. It is in the lineage of Weierstrass' example of a continuous nowhere differentiable function.

We define $f$ via a lacunary Fourier series
\begin{equation}
\label{E:f}
 f(x) = \sum_{k=1}^{\infty}a_{k}\sin(b_{k}x) = \sum_{k=1}^{\infty}f_{k}(x)
\end{equation}
where the sequences $a_{k}>0$ and $b_{k}\in\N$ are chosen so that they satisfy 
$\sum a_{k}<\infty$ and $b_{k}\ge2$, 
$b_{k+1}\ge2b_{k}$, $k\ge1$, and the following properties for a given small constant $\gamma>0$ to be fixed later on (see \eqref{E:gamma}):
\begin{align}
\label{E:c-}
 \sum_{k=1}^{m-1}a_{k}b_{k} &\le \gamma\, a_{m}b_{m} &\quad\forall\,m\ge2\\
\label{E:c+}
 \sum_{k=m+1}^{\infty}a_{k} &\le \gamma\, a_{m} &\quad\forall\,m\ge1\\
\label{E:suminf}
 \sum_{m=1}^{\infty} a_{m}^{p}b_{m}^{\,p\varepsilon} &= +\infty
 &\quad\forall\,\varepsilon>0, \:p\ge1 \,.
\end{align}
We first show that for sufficiently large $q\in\N$ the sequences
$a_{k}=q^{-k}$, $b_{k}=2^{q^{k}}$
have the properties \eqref{E:c-}--\eqref{E:suminf}, and we shall keep this choice from now on.

For \eqref{E:c-}, let 
$s_{m}=\frac1{a_{m}b_{m}}\sum_{k=1}^{m-1}a_{k}b_{k}$. 
Noting that for $q\ge7$ we have 
$q^{2}\,2^{1-q}<1$, we show by induction that then $s_{m}<\frac1{q-1}$ for all $m\ge2$, which implies \eqref{E:c-} for $q$ large enough.
Indeed, 
$$
  s_{2}=\tfrac{a_{1}b_{1}}{a_{2}b_{2}}=q\,2^{(1-q)q}
  <q\,2^{1-q}<\tfrac1q<\tfrac1{q-1}\,,
$$ 
and if $s_{m}<\frac1{q-1}$ it follows that
$$
 s_{m+1}=(s_{m}+1)\tfrac{a_{m}b_{m}}{a_{m+1}b_{m+1}}
        =(s_{m}+1)\,q\,2^{(1-q)q^{m}} < (s_{m}+1)\,q\,2^{(1-q)}
        <(\tfrac{1}{q-1}+1)\tfrac1q
        =\tfrac1{q-1}\,.
$$

For \eqref{E:c+}, we have
$$
  \sum_{k=m+1}^{\infty}\frac{a_{k}}{a_{m}} = \sum_{k=1}^{\infty}q^{-k}
   = \frac{1}{q-1}
$$
which again is less than $\gamma$ for $q$ large enough.

For \eqref{E:suminf} we use that $2^{t}\ge t\log2$ for all $t>0$, so that
$
a_{m}^{p}b_{m}^{\,p\varepsilon} = (2^{\eps q^{m}}/q^{m})^{p}
 \ge (\eps\log2)^{p}
$ 
for all $m$.

\begin{lemma}
\label{L:fnotHeps}
 The function $f$ defined by \eqref{E:f} is continuous on $\T$ and satisfies
\begin{equation}
\label{E:fnotHeps}
 \int_{0}^{2\pi}\frac{|f(y)-f(x)|^{p}}{|y-x|^{1+p\varepsilon}}dy
 = +\infty
 \qquad \mbox{ for all } x\in[0,2\pi],\; \varepsilon>0, 1\le p<\infty\,.
\end{equation}
\end{lemma}
\begin{proof}
Noting that with our even $b_{k}$ we have $f(2\pi-x)=f(x)$, so that it is sufficient to prove \eqref{E:fnotHeps} for $x\in[0,\pi]$. In this case $[x,x+1]\subset[0,2\pi]$, and therefore 
with $I_{m}=[\frac1{b_{m}},\frac2{b_{m}}]$ we have
\begin{equation}
\label{E:sumIm}
 \int_{0}^{2\pi}\frac{|f(y)-f(x)|^{p}}{|y-x|^{1+p\varepsilon}}dy
 \ge
 \sum_{m=1}^{\infty}\int_{I_{m}}\frac{|f(x+h)-f(x)|^{p}}{|h|^{1+p\varepsilon}}dh
\end{equation}
Now for $h\in I_{m}$ we estimate 
$$
 \Big(\int_{I_{m}}\frac{|f(x+h)-f(x)|^{p}}{|h|^{1+p\varepsilon}}dh\Big)^{\frac1p}
 \ge  J_{1} -J_{2}
$$
with
$\di
 J_{1}= \Big(\int_{I_{m}}\frac{|f_{m}(x+h)-f_{m}(x)|^{p}}{|h|^{1+p\varepsilon}}dh\Big)^{\frac1p}
$
and
$\di
 J_{2}=
 \sum_{k\ne m}\Big(\int_{I_{m}}\frac{|f_{k}(x+h)-f_{k}(x)|^{p}}{|h|^{1+p\varepsilon}}dh\Big)^{\frac1p}\,.
$

To estimate $J_{1}$,
we assume that $0<\varepsilon<1$ 
and make the change of variables $t=b_{m}h$ to obtain
$$
 J_{1} = a_{m}b_{m}^{\varepsilon}
  \Big(\int_{1}^{2}|\sin(b_{m}x+t)-\sin(b_{m}x)|^{p}t^{-(1+p\varepsilon)}
 dt\Big)^{\frac1p} \ge 5\,\gamma\,a_{m}b_{m}^{\,\varepsilon}\, ,
$$
where we defined 
\begin{equation}
\label{E:gamma}
 \gamma = \tfrac15
  \min_{z\in\T}
  \int_{1}^{2}|\sin(z+t)-\sin(z)|t^{-2}dt>0\,.
\end{equation}
Here we used H\"older's inequality,
\begin{align*}
  \int_{1}^{2}\frac{|\sin(z+t)-\sin(z)|}{t^{2}}dt
  &\le \int_{1}^{2}\frac{|\sin(z+t)-\sin(z)|}{t^{1+\eps}}dt\\
  &\le \Big(\int_{1}^{2}|\sin(z+t)-\sin(z)|^{p}\,t^{-(1+p\varepsilon)}
 dt\Big)^{\frac1p}\,
 \Big(\int_{1}^{2}\frac{dt}t\Big)^{1-\frac1p}\,.
\end{align*}

To estimate $J_{2}$,
we use for $k\le m-1$
$$
 |f_{k}(x+h)-f_{k}(x)| \le a_{k}b_{k}|h| \le 2a_{k}b_{k}\tfrac1{b_{m}} 
$$
and for $k\ge m+1$
$$
 |f_{k}(x+h)-f_{k}(x)| \le 2 a_{k} 
$$
so that we obtain with \eqref{E:c-}
$$
 \sum_{k=1}^{m-1}\Big(\int_{I_{m}}\frac{|f_{k}(x+h)-f_{k}(x)|^{p}}{|h|^{1+p\varepsilon}}dh\Big)^{\frac1p} \le 
 2\gamma a_{m}\Big(\int_{I_{m}}\frac{dh}{|h|^{1+p\varepsilon}}\Big)^{\frac1p}
 \le 2\gamma a_{m} b_{m}^{\,\eps}
$$
and with \eqref{E:c+}
$$
 \sum_{k=m+1}^{\infty}\Big(\int_{I_{m}}\frac{|f_{k}(x+h)-f_{k}(x)|^{p}}{|h|^{1+p\varepsilon}}dh\Big)^{\frac1p} \le
 2\gamma a_{m}\Big(\int_{I_{m}}\frac{dh}{|h|^{1+p\varepsilon}}\Big)^{\frac1p}
 \le 2\gamma a_{m} b_{m}^{\,\eps}\,,
$$
hence  
$J_{2} \le 4\gamma a_{m} b_{m}^{\,\eps}\,.$

Together, this gives
$$
 \Big(\int_{I_{m}}\frac{|f(x+h)-f(x)|^{p}}{|h|^{1+p\varepsilon}}dh\Big)^{\frac1p}
 \ge \gamma\,a_{m}b_{m}^{\varepsilon}\,, 
$$
and finally with \eqref{E:sumIm} and \eqref{E:suminf}
$$
 \int_{0}^{2\pi}\frac{|f(y)-f(x)|^{p}}{|y-x|^{1+p\varepsilon}}dy
 \ge \sum_{m=1}^{\infty} \gamma^{p} a_{m}^{p} b_{m}^{\,p\varepsilon} = +\infty\,.
$$
\end{proof}

\begin{proposition}
\label{P:sep}
 The function $f$ defined by \eqref{E:f} has the following separation property: Let $0<\varepsilon<1$, $p\ge1$ and $a,b\in W^{\eps,p}(0,2\pi)$. If $\,af=b$, then $a=b=0$.
\end{proposition}
\begin{proof}
Write the $W^{\varepsilon,p}$ seminorm as in \eqref{E:sp-seminorm}
$$
 |b|_{\varepsilon,p}=
  \Big(\int_{0}^{2\pi}\int_{0}^{2\pi}\frac{|b(y)-b(x)|^{p}}{|y-x|^{1+p\varepsilon}}dy\,dx\Big)^{\frac1p}\,.
$$
Using
$$
 b(y)-b(x)=(f(y)-f(x))a(x) + f(y)(a(y)-a(x))
$$
and the triangle inequality, we find for $a,b\in W^{\eps,p}(0,2\pi)$
$$ 
 \Big(\int_{0}^{2\pi}\int_{0}^{2\pi}\frac{|a(x)|^{p}\,|f(y)-f(x)|^{p}}{|y-x|^{1+p\varepsilon}}dy\,dx\Big)^{\frac1p} 
 \le
 |b|_{\varepsilon,p} + \|f\|_{L^{\infty}(\T)}|a|_{\varepsilon,p} < \infty\,.
$$
Because of \eqref{E:fnotHeps} from Lemma~\ref{L:fnotHeps}, this implies $a(x)=0$ for almost all $x\in\T$ and then $b=af=0$.

\end{proof}

\section{2D domain with limited regularity}\label{S:d=2}

Let $F(x)=1+\int_{0}^{x}f(t)dt$. 
Then $F\in C^{1}(\T)$, $F'=f$, and $\frac12<F(x)<\frac32$.

The latter estimate follows easily from 
$$
 |F(x)-1|=|\sum_{k=1}^{\infty}a_{k}\tfrac{1-\cos(b_{k}x)}{b_{k}}|
 \le
 2^{-q}\sum_{k=1}^{\infty}2\,q^{-k} = 2^{1-q}\tfrac1{q-1} \le \frac12\,.
$$
We define now the $C^{1}$ domain $\omega\subset\R^{2}$ using polar coordinates $(r,\theta)$
$$
 \omega  = \{(r,\theta) \mid r<F(\theta)\}\,.
$$
\begin{proposition}
\label{P:normaltrace}
 Let $p\ge1$, $\eps>0$ and $u\in W^{\frac1p+\varepsilon,p}(\omega;\C^{2})$ be such that its normal trace
 $n\cdot u$ vanishes on $\partial\omega$. Then $u=0$ on $\partial\omega$. The same conclusion is valid when the tangential trace $n\times u$ vanishes on $\partial\omega$. 
\end{proposition}
\begin{proof} (Following Filonov \cite[\S5]{Filonov97})
The unit normal $n$ on $\partial\omega$ has the Cartesian components
$$
  n_{1}=(F^{2}+f^{2})^{-\frac12}(F\cos\theta + f\sin\theta),\quad
  n_{2}=(F^{2}+f^{2})^{-\frac12}(F\sin\theta - f\cos\theta)\,.
$$
Therefore the condition $n_{1}u_{1}+n_{2}u_{2}=0$ implies $af=b$ if we define
$$
 a=u_{2}\cos\theta-u_{1}\sin\theta\,,\quad
 b=(u_{1}\cos\theta+u_{2}\sin\theta)F
$$
Now, since the traces $u_{j}$ on $\partial\omega$, understood as functions 
$\theta\mapsto u_{j}(F(\theta),\theta)$ on $\T$, belong to $W^{\varepsilon,p}(\T)$, we also have 
$a,b\in W^{\varepsilon,p}(\T)$. According to Proposition~\ref{P:sep} we find $a=b=0$, which implies $u_{1}=u_{2}=0$ on $\partial\omega$.
The result using vanishing tangential trace follows by a rotation by $\pi/2$.
\end{proof}

\begin{corollary}
\label{C:dirneu}
(i) There exists $g\in C^{\infty}(\overline\omega)$ such that the solution $v_{D}\in H^{1}_{0}(\omega)$ of the Dirichlet problem
$$
 \Delta v_{D}=g \;\mbox{ in }\omega\,;\quad v_{D}=0 \; \mbox{ on }\partial\omega
$$
does not belong to $W^{1+\frac1p+\varepsilon,p}(\omega)$ for any $\epsilon>0$, $p\ge1$.\\
(i) There exists $g\in C^{\infty}(\overline\omega)$ such that any solution $v_{N}\in H^{1}(\omega)$ of the Neumann problem
$$
 \Delta v_{N}=g \;\mbox{ in }\omega\,;\quad \partial_{n}v_{N}=0 \; \mbox{ on }\partial\omega
$$
does not belong to $W^{1+\frac1p+\varepsilon,p}(\omega)$ for any $\varepsilon>0$, $p\ge1$.\\
\end{corollary}
\begin{proof}
For $v_{D}$ one can take $g=1$. Set $u=\nabla v_{D}$. If 
$v_{D}\in W^{1+\frac1p+\varepsilon,p}(\omega)$, then $u$ satisfies the hypotheses of Proposition~\ref{P:normaltrace} with vanishing tangential trace. Hence also the normal trace of $u$ vanishes, i.e. $\partial_{n}v_{D}=0$ on $\partial\omega$. Then Green's formula implies $\int_{\omega}g=0$, which is not the case.

For $v_{N}\in W^{1+\frac1p+\varepsilon,p}(\omega)$ one obtains similarly that the tangential derivative on the boundary vanishes, hence the trace of $v_{N}$ on $\partial\omega$ is constant, without loss of generality equal to zero. Thus $v_{N}$ is also solution of the Dirichlet problem. That there exists $g\in L^{2}(\omega)$ for which this is impossible can be seen as follows:

Let $g$ be a non-zero harmonic polynomial such that $\int_{\omega}g=0$, for example $g(x_{1},x_{2})=\alpha x_{1}x_{2} +\beta (x_{1}^{2}-x_{2}^{2})$ with suitably chosen coefficients $\alpha,\beta\in\R$. Then $v_{N}$ exists, and Green's formula gives the contradiction
$$
 0=\int_{\partial\omega}(\partial_{n}v_{N}\,g-v_{N}\partial_{n}g)ds
  =\int_{\omega}(\Delta v_{N}\, g- v_{N}\Delta g)dx = \int_{\omega}g^{2}dx \,.
$$
\end{proof}

\begin{remark}
\label{R:DirEF} 
No eigenfunction of the Laplacian with Dirichlet conditions on $\omega$ can belong to $W^{1+\frac1p+\varepsilon,p}(\omega)$ with $\varepsilon>0$, because it would also have vanishing normal derivative. Its extension by zero outside $\omega$ would then be a Dirichlet eigenfunction with the same eigenvalue on any domain containing $\omega$. This contradicts for example the well known behavior of Dirichlet eigenvalues on disks or squares with varying size. It contradicts also the well known interior analyticity of Dirichlet eigenfunctions.
\end{remark} 

\section{Example in higher dimensions}\label{S:d>2}
From $\omega\subset\R^{2}$ one can construct $\Omega\subset\R^{d}$ as follows (see \cite{Filonov97}, for $n=3$ also \cite[\S6]{BuCoSh02}). In cylindrical coordinates $(r,\theta,z)$, $z\in\R^{d-2}$:
$$
 \Omega = \{ (r,\theta,z) \mid \frac{r^{2}}{F(\theta)^{2}} + |z|^{2}<1\}
$$
The intersection with the plane $z=z_{0}$ gives for $|z_{0}|<1$ the scaled domain $\sqrt{1-|z_{0}|^{2}}\,\omega$.
One can still prove that for this domain $\Omega$ and $0<\epsilon<1$ there holds 
\begin{equation}
\label{E:W0Omega}
 W^{1+\frac1p+\varepsilon,p}(\Omega)\cap W^{1,p}_{0}(\Omega)=W^{1+\frac1p+\varepsilon,p}_{0}(\Omega)\,.
\end{equation}
Indeed, suppose that $v\in W^{1+\frac1p+\varepsilon,p}(\Omega)$, $v=0$ on $\partial\Omega$ and let $u=\nabla v$. Then the tangential components of $u$ are zero on the boundary, and we have to show that the normal component of $u$ vanishes, too, on $\partial\Omega$. Define
$$
  \tilde u(r,\theta,z)=u(\sqrt{1-|z|^{2}}\,r,\theta,z)\,.
$$
Then $\tilde u$ is defined on the product domain
$$
 \tilde\Omega = \omega\times B_{1} = \{(r,\theta,z)\mid (r,\theta)\in\omega, |z|<1\}\,.
$$
For any $\delta\in(0,1)$, let $\tilde\Omega_{\delta}=\omega\times B_{\delta}$. Then $\tilde u$ restricted to $\tilde\Omega_{\delta}$ belongs to 
$$
 W^{\frac1p+\eps,p}(\tilde\Omega_{\delta};\C^{d}) 
 \subset L^{p}\big(B_{\delta};W^{\frac1p+\eps,p}(\omega;\C^{d})\big)\,,
$$
and for almost every $z_{0}\in B_{\delta}$, the restriction $w_{z_{0}}$ of $\tilde u$ to the plane $z=z_{0}$ belongs to $W^{\frac1p+\eps,p}(\omega,\C^{d})$. The vanishing of the tangential components of $u$ on $\partial\Omega$ implies that the component of $w_{z_{0}}$ that is parallel to the plane $z=0$ and tangential to $\partial\omega$ vanishes on $\partial\omega$. Then Proposition~\ref{P:normaltrace} tells us that the component of $w_{z_{0}}$ that is parallel to the plane $z=0$ and normal to $\partial\omega$ vanishes on $\partial\omega$, too. This means that at such a point $(r,\theta,z)\in\partial\Omega$ with 
$(\sqrt{1-|z|^{2}}\,r,\theta)\in\partial\omega$, $z=z_{0}$, in addition to the tangential components a component of $u$ vanishes that is not tangential, and hence all components of $u$ vanish there. Since this is true for almost all $z_{0}$ satisfying $|z_{0}|<\delta$ and for all $0<\delta<1$, we see that the trace of $u$ on $\partial\Omega$ is zero, which proves \eqref{E:W0Omega}.

The non-regularity result of Theorem~\ref{T:main} for the Dirichlet problem in $\Omega$ then follows in the same way as in the two-dimensional case. In particular, one can take $g=1$ for the counterexample.

For the Neumann problem, a slightly different variant of adding $d-2$ variables works, and this variant could also be used for the Dirichlet problem, giving a counterexample with a somewhat less regular right hand side $g$. For this variant, \eqref{E:W0Omega} still holds. We redefine the domain $\Omega$ so that it contains a cylindrical part (see also \cite[\S5.2]{Filonov97}). This is done by modifying the function $1-|z|^{2}$ in the previous example. Choose a decreasing $C^{\infty}$ function $\mu$ on $\R_{+}$ satisfying
$$
 \mu(t)=1\; \mbox{ for } t\le1\,;\qquad
 \mu(t)\le0 \; \mbox{ for } t\ge4\,;\qquad
 \mu'(t)<0 \; \mbox{ for } t\ge2\,.
$$
and define
\begin{equation}
\label{E:Omegad}
 \Omega = \{ (r,\theta,z) \mid r^{2}< \mu(|z|^{2})\,F(\theta)^{2}\}\,.
\end{equation}
It is not hard to see that $\Omega$ has a $C^{1}$ boundary.

We now use the two-dimensional example presented in the previous section and denote by $v_{0}$ the function found there that satisfies the Neumann problem on $\omega$ with right hand side $g_{0}\in C^{\infty}(\overline\omega)$ and that does not belong to any $W^{1+\frac1p+\eps,p}(\omega)$ for $\eps>0$, $p\ge1$.
In addition, we choose a function $\chi\in C^{\infty}_{0}(\overline{\R_{+}})$ satisfying
$\chi(t)=1$ for $t<\frac12$, $\chi(t)=0$ for $t\ge1$.
Then we define
$$
  v(x,z)=v_{0}(x)\,\chi(|z|);\qquad
  g(x,z)=g_{0}(x)\,\chi(|z|) + v_{0}(x)\Delta_{z}\chi(|z|);\qquad
  (x\in\omega,\; |z|<1)\,.
$$
Initially, $v$ and $g$ are defined on the cylinder $\omega\times B_{1}\subset\Omega$, and we extend them by zero on the rest of $\Omega$.

One easily verifies that $v$ satisfies 
$$
  \Delta v = g\;\mbox{ in }\Omega\,;\qquad
  \partial_{n}v=0\;\mbox{ on }\partial\Omega\,.
$$
Noting that both $\chi(|z|)$ and $\Delta_{z}\chi(|z|)$ define $C^{\infty}(\overline\Omega)$ functions and using the regularity of 
$v_{0}\in W^{1+\frac1p,p}(\omega)$ for all $p>1$, so that $v_{0}$ is H\"older continuous on $\overline\omega$, one finds that $g$ is H\"older continuous on $\overline\Omega$. Finally the non-regularity of $v_{0}$ implies clearly that also $v\not\in $$W^{1+\frac1p+\eps,p}(\Omega)$ for $\eps>0$, $p\ge1$.

This concludes the proof of Theorem~\ref{T:main}.



\end{document}